\documentclass[12pt]{amsart}

\usepackage{mathptmx}
\usepackage{setspace}
\usepackage{amsmath,amsthm,amssymb,latexsym,amscd}
\usepackage{amscd}
\usepackage{amssymb}
\usepackage{amsmath,amssymb}
\usepackage{amssymb,epsfig}
\usepackage[all]{xy}
\usepackage{pst-plot}
\usepackage{amsmath,verbatim}
\usepackage{subfigure}

\newcommand{\nc}{\newcommand}
\nc{\nt}{\newtheorem}
\nc{\bs}{\bigskip}
\nc{\dmo}{\DeclareMathOperator}

\nt{main}{Main Result}
\nt{thm}{Theorem}[section]
\nt{prop}[thm]{Proposition}
\nt{lem}[thm]{Lemma}
\nt{fact}[thm]{Fact}
\nt{cor}[thm]{Corollary}
\nt{conj}[thm]{Conjecture}
\nt{question}[thm]{Question}
\nt{problem}[thm]{Problem}
\nt*{mainproblem}{Problem}
\nt{claim}[thm]{Claim}

\nc{\I}{\mathcal{I}}
\nc{\C}{\mathcal{C}}
\nc{\T}{\mathcal{T}}
\nc{\A}{\mathcal{A}}
\nc{\B}{\mathcal{B}}
\nc{\K}{\mathcal{K}}
\nc{\Y}{\mathcal{Y}}
\nc{\SI}{\mathcal{SI}}
\nc{\SK}{\mathcal{SK}}
\nc{\ST}{\mathcal{ST}}
\nc{\SIBK}{\mathcal{SIBK}}
\nc{\SBK}{\mathcal{SBK}}
\nc{\SC}{\mathcal{SC}}
\nc{\SL}{\mathcal{SL}}
\nc{\SN}{\mathcal{SN}}
\dmo{\Mod}{Mod}
\dmo{\PMod}{PMod}
\dmo{\SIP}{SIP}
\dmo{\SMod}{SMod}
\dmo{\Sp}{Sp}
\dmo{\SSp}{SSp}
\dmo{\Stab}{Stab}
\dmo{\Isom}{Isom}
\dmo{\Teich}{Teich}

\nc{\ia}{\hat i}
\nc{\Z}{\mathbb{Z}}
\nc{\R}{\mathbb{R}}

\begin{document}
\title{SIMPLY INTERSECTING PAIR MAPS IN THE MAPPING CLASS GROUP}
\author{LEAH R. CHILDERS }
\thanks{This paper is part of the author's PhD dissertation at Louisiana State University.  The author is deeply grateful to the guidance of her advisor, Tara Brendle.  The author would also like to thank the University of Glasgow and University of Chicago for their generous hospitality while working on this paper.  This research was supported by the VIGRE grant at LSU, DMS-0739382. The author would also like to thank Joan Birman, Tara Brendle, Tom Church, Pallavi Dani, Benson Farb, Charlie Egedy, Rolland Trapp, Chris Leininger, and Dan Margalit for their discussions about the ideas in this paper.} 

\begin{abstract}
The Torelli group, $\I(S_g)$, is the subgroup of the mapping class group consisting of elements that act trivially on the homology of the surface.   There are three types of elements that naturally arise in studying $\I(S_g)$: bounding pair maps,  separating twists, and simply intersecting pair maps (SIP-maps).  Historically the first two types of elements have been the focus of the literature on $\I(S_g)$, while SIP-maps have received relatively little attention until recently, due to an infinite presentation of $\I(S_g)$ introduced by Andrew Putman that uses all three types of elements.  We will give a topological characterization of the image of an SIP-map under the Johnson homomorphism and Birman-Craggs-Johnson homomorphism.  We will also classify which SIP-maps are in the kernel of these homomorphisms.  Then we will look at the subgroup generated by all SIP-maps, $\SIP(S_g)$, and show it is an infinite index subgroup of $\I(S_g)$.  
\end{abstract}

\maketitle
\section{Introduction}
Let $S_{g,b,n}$ be an oriented surface of genus $g$ with $b$ boundary components and $n$ punctures.  Our convention is that boundary components are always fixed pointwise. Further we will often omit an index if it is $0$ (sometimes $1$, when noted as such). We define the \emph{mapping class group of $S_g$}, $\Mod(S_g)$, to be:
$$ \Mod(S_g) := \textrm{Homeo}^+(S_g)/\textrm{Homeo}^+_0(S_g)$$
where $\textrm{Homeo}^+(S_g)$ is the group of of orientation preserving homeomorphisms of $S_g$ and $\textrm{Homeo}^+_0(S_g)$ is the normal subgroup consisting of elements isotopic to the identity. Thus $\Mod(S_g)$ is the group of isotopy classes of orientation preserving self-homeomorphisms of a surface. See \cite{birman}, \cite{primer}, and \cite{ivanovsurvey} for background information. A subgroup of the mapping class group of primary importance is the Torelli group, $\I(S_g)$,  the kernel of the well-known \emph{symplectic representation} of the mapping class group.  Mapping classes act naturally on the first homology of the surface and preserve the intersection form, giving rise to a surjective map to $\text{Sp}(2g, \Z)$ (see Chapter 7 of \cite{primer}).  This action is known as the \emph{symplectic representation}.
 $$1 \rightarrow \I(S_g) \rightarrow \Mod(S_g) \rightarrow \text{Sp}(2g,\Z) \rightarrow 1$$
Equivalently, $\I(S_g)$ is the subgroup of $\Mod(S_g)$ acting trivially on the homology of the surface.  

There are three types of elements that naturally arise in studying $\I(S_g)$: bounding pair maps (BP-maps), Dehn twists about separating curves, and simply intersecting pair maps (SIP-maps).  Historically the first two types of elements have been the focus of the literature on $\I(S_g)$.   For example, in \cite{johnson1}, \cite{johnsonabelian}, and \cite{johnson2} Johnson showed that when $g \geq 3$ BP-maps generate $\I(S)$ and further that Dehn twists about separating curves generate an infinite index subgroup of $\I(S_g)$, called the \emph{Johnson kernel}, $\K(S)$.  However, SIP-maps have been brought to the forefront due to an infinite presentation of $\I(S_g)$ introduced by Putman that uses all three types of elements \cite{putmaninfinite}.   While Putman's presentation includes an infinite number of generators, he is able to classify all relations amoung these generators into a finite number of classes.  Thus the hope is that a better understanding of SIP-maps might yield insight into answering the question: Is $\I(S_g)$ finitely presentable when $g \geq 3$?

The purpose of this paper is to investigate SIP-maps in their own right.  Moreover, we will compare and contrast SIP-maps with the standard elements of $\I(S_g)$, that is with BP-maps and Dehn twists about separating curves.  In Section~\ref{section:basicfacts}, we will begin our discussion of SIP-maps by showing

\begin{main}[Corollary~\ref{cor:pseudoanosov}]
Let the commutator $f = [T_a,T_b]$ be an SIP-map with associated lantern $L$.  Then $f$ is pseudo-Anosov on $L$.
\end{main}

In addition, we will discuss how the same two curves can yield several distinct SIP-maps.  Then in Section~\ref{section:sipgroup}, we consider the group generated by SIP-maps, which we call $\SIP(S_g)$.  In light of Johnson's work regarding BP-maps and separating twists, it is natural to ask: Is $\SIP(S_g) = \I(S_g)?$  If not, what is the index of  $\SIP(S_g)$ in $\I(S_g)$? 

We will answer these questions by giving a topological description of the image of SIP-maps under the \emph{Johnson homomorphism}, $\tau : \I(S_{g,1}) \longrightarrow \wedge^3 H$ where $H = H_1(S_{g,1}, \Z)$.  For the most part we will simply think of $H$ as an abelian group.  The Johnson homomorphism is one of the classical abelian quotients of $\I(S_{g,1})$.  The abelianization of $\I(S_{g,1}) \cong \wedge^3 H \oplus B$ where $B$ is a number of copies of $\Z_2$.  Of the two pieces of the abelianization of $\I(S_{g,1})$, one is captured by $\tau$ and the other by the Birman-Craggs-Johnson homomorphism $\sigma$.  For that reason we will calculate the image of an SIP-map under both these homomorphisms.

In order to do this calculation we first rewrite SIP-maps as the product of five BP-maps.  We use this rewriting to show:

\begin{main}[Proposition~\ref{prop:johnsonsip}]
The image under $\tau$ of an SIP-map whose associated lantern has boundary components $w, x, y,$ and $z$ is $\pm x \wedge y \wedge z$.  
\end{main}

We will use this same rewriting to find a topological description of the image under $\sigma$ of an SIP-map.  Further, from these calculations, we are also able to deduce that $\SIP(S_g) \neq \I(S_g)$ by observing that SIP-maps are in the kernel of the so-called ``contraction map."  Further, the contraction map shows:

\begin{main}[Corollary~\ref{cor:infiniteindex}]
The group $\SIP(S_g)$ is an infinite index subgroup of $\I(S_g)$.
\end{main}

In addition, we are able to characterize which SIP-maps are in ker $\tau = \K(S)$.

\begin{main}[Corollary~\ref{cor:sipkernel}]
An SIP-map $f$ is an element of $\K(S)$ if and only if the lantern associated with $f$ has a boundary component that is null-homologous or if two boundary components are homologous. 
\end{main}

In Section \ref{section:relations} we will show how SIP-maps can be used to reinterpret one of Johnson's relations about BP-maps.  

Then in Section \ref{section:bcj} we characterize the image of SIP-maps under the \emph{Birman-Craggs-Johnson homomorphism}, $\sigma: \I \rightarrow B_3$, where $B_3$ is a $\Z_2$-vector space of Boolean (square free) polynomials with generators corresponding to non-zero elements of $H_1(S,\Z_2)$ \cite{johnsonbcj}.  We show:
 
\begin{main}[Proposition~\ref{prop:sipbcj} and Corollary~\ref{cor:sipbcj}]
Let $\sigma$ be the Birman-Craggs-Johnson homomorphism.  Let $f$ be an SIP-map with associated lantern $L$.  If none of the boundary components of L are null-homologous, then $\sigma(f)$ is a cubic polynomial in the homology classes of the boundary components of $L$.  Further $\sigma(f) = 0$ if and only if one of the boundary components of $L$ is null-homologous.
\end{main}
Since $\SIP(S_g)$ is not $\I(S_g)$, it is natural to ask:  What is the precise structure of $\SIP(S_g)$?  For example, since we know the abelian quotient of $\I(S_g)$ is captured by $\tau$ and $\sigma$, we can ask the following:

\begin{mainproblem}
\label{abelian}
Is $\I(S_g)/\SIP(S_g)$ abelian?
\end{mainproblem}
Building on work of Johnson, to answer Problem~\ref{abelian} it will suffice to establish if the intersection of the Johnson kernel and the Birman-Craggs-Johnson kernel lies in $\SIP(S_g)$ \cite{johnson3}.   In the main results above, we have characterized which SIP-maps are in the Johnson kernel, $\K(S)$, and which are in the Birman-Craggs-Johnson kernel. It remains to investigate the converse:

\begin{mainproblem}
Which elements of the Johnson kernel, $\K(S_g)$, and Birman-Craggs-Johnson kernel lie in $\SIP(S_g)$?
\end{mainproblem}

While Johnson has given a completely algebraic characterization of the Birman-Craggs-Johnson kernel \cite{johnson3}, this kernel is still not well understood in terms of BP-maps, separating twists, and SIP-maps, all of which have a natural topological structure. 
 
\section{Background}
\label{section:background}
\textbf{Basic Definitions.} We will refer to a simple closed curve as a \emph{curve} unless stated otherwise and we will often not distinguish between a curve and its isotopy class unless needed.

The simplest infinite order element in $\Mod(S)$ is a (right) \emph{Dehn twist} about a simple closed curve $c$, denoted $T_c$.  One can think of this map as cutting the surface along $c$ and twisting a neighborhood of one of the boundary components $360^{\circ}$, and then gluing the surface back together along $c$.  For example, in Figure~\ref{figure:dehntwist}, we see the image of the curve $d$ under the mapping class $T_c$.

\begin{figure}[htb]
\centerline{\includegraphics[scale=.5]{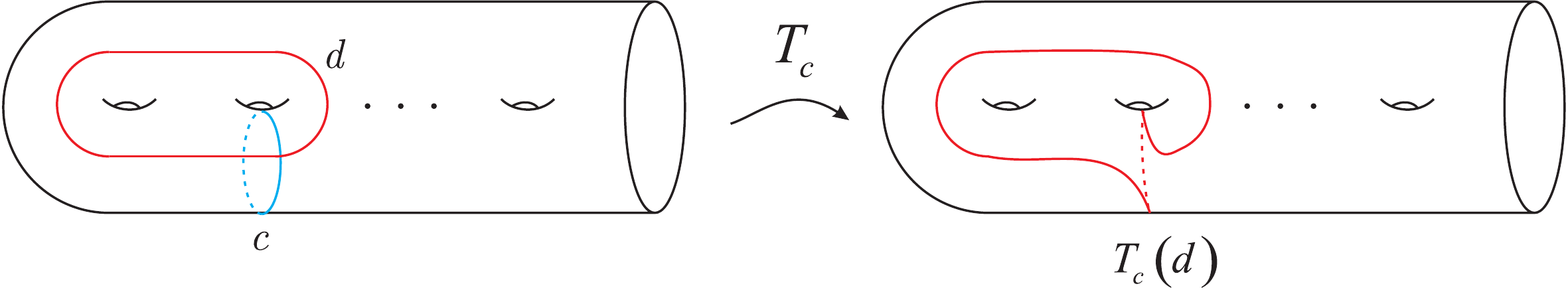}}
\caption{An example of the image of a curve under a Dehn twist.}
\label{figure:dehntwist}
\end{figure}

For completeness, note that Dehn twists are basic elements of the mapping class group in the following sense.

\begin{thm}[Dehn, \cite{dehn}]
The mapping class group, $\Mod(S)$, is generated by finitely many Dehn twists about simple closed curves.
\end{thm}

The \emph{algebraic intersection number} of a pair of transverse, oriented curves $\alpha$ and $\beta$ on a surface, denoted $\hat{i}(\alpha,\beta)$, is the sum of the indices of the intersection points of $\alpha$ and $\beta$, where an intersection point has index $+1$ if the orientation of the intersection agrees with the orientation of the surface, and $-1$ otherwise.

The \emph{geometric intersection number} of a pair of curves $\alpha$ and $\beta$ is defined as
$$ i(\alpha, \beta) = \displaystyle \min_{a \in \alpha, \, b \in \beta} |a \cap b|.$$
Note this is well-defined on isotopy classes of curves (Chapter 1, \cite{primer}).

\textbf{Relations in $\Mod(S)$. } We will discuss several well-known relations in $\Mod(S)$ that will be used throughout this paper, most notably the so-called \emph{lantern relation}.  Proofs for all these can be found in Chapter 2 of \cite{primer}.

\begin{lem}
\label{lem:conjugationrelation}
Let $f \in \Mod(S)$ and $a$ be a curve on $S$.  Then $f T_a f^{-1} = T_{f(a)}$.
\end{lem}

\begin{lem}
\label{lem:isotopyrelation}
Let $a$ and $b$ be curves on $S$.  Then $T_a = T_b$ if and only if $a$ is isotopic to $b$.
\end{lem}

\textbf{Lantern Relation in $\Mod(S)$. } The \emph{lantern relation} is a relation in $\Mod(S)$ among 7 Dehn twists all supported on a subsurface of $S$ homeomorphic to a sphere with 4 boundary components (otherwise known as a lantern).  This relation was known to Dehn \cite{dehn}, and later rediscovered by Johnson \cite{johnsonhomeos}.  The lantern relation will be particularly important in Lemma~\ref{lem:factorsip} when writing an $\SIP$-map as the product of BP-maps.  Given curves $a, b, c,$ and $d$, so that $a, b, c,$ and $d$ bound a lantern, then the following relation holds where the curves are as in Figure~\ref{figure:lantern}.
$$T_aT_bT_cT_d = T_xT_yT_z.$$

\begin{figure}[htb]
\centerline{\includegraphics[scale=.5]{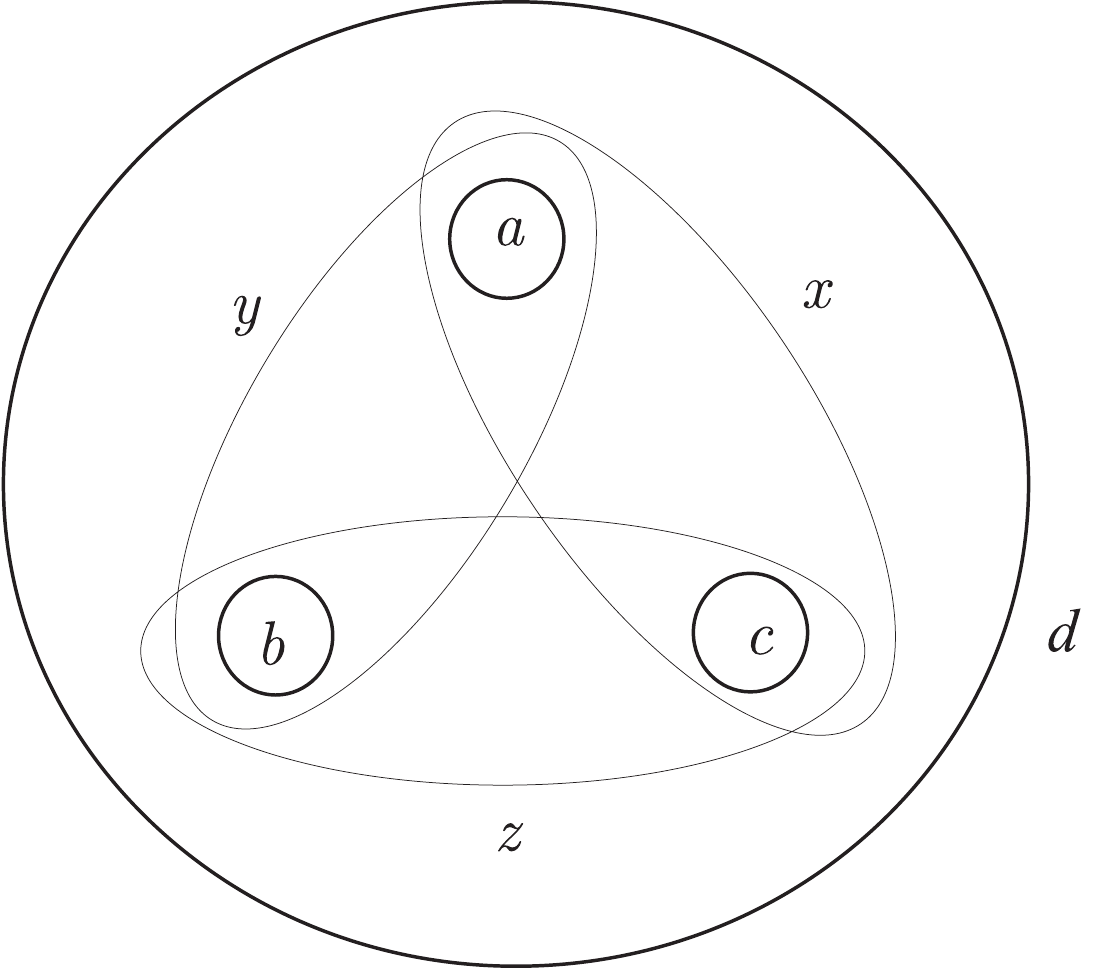}}
\caption{The curves in the lantern relation: $T_aT_bT_cT_d = T_xT_yT_z$}
\label{figure:lantern}
\end{figure}

\textbf{The Torelli Group.} A subgroup of the mapping class group of primary importance is the \emph{Torelli group}, $\I(S_g)$,  the kernel of the well-known \emph{symplectic representation} of the mapping class group.
 $$1 \rightarrow \I(S_g) \rightarrow \Mod(S_g) \rightarrow \text{Sp}(2g,\Z) \rightarrow 1$$
Equivalently, $\I(S_g)$ is the subgroup of $\Mod(S_g)$ acting trivially on the homology of the surface. Note that we will often refer to $H_1(S, \Z)$ simply as $H$. Further, because the symplectic group, $\text{Sp}(2g, \Z)$ is well understood, $\I(S)$ is often thought of as the ``mysterious" part of $\Mod(S)$.  Further when $g = 1$ the symplectic representation is faithful, so $\I(S) = 1.$  

There are three types of elements that naturally arise in studying $\I(S)$:

\begin{enumerate}
\item \textbf{Bounding Pair Maps.}  Given two disjoint, non-separating, homologous simple closed curves $c$ and $d$, a \emph{bounding pair map} (BP-map) is the product $T_cT_d^{-1}$.  If $S = S_{g,1}$, then we say a BP-map has \emph{genus k} if the subsurface whose boundary is $c \cup d$ has genus $k$.\\ 

\begin{figure}[htb]
\centerline{\includegraphics[scale=.4]{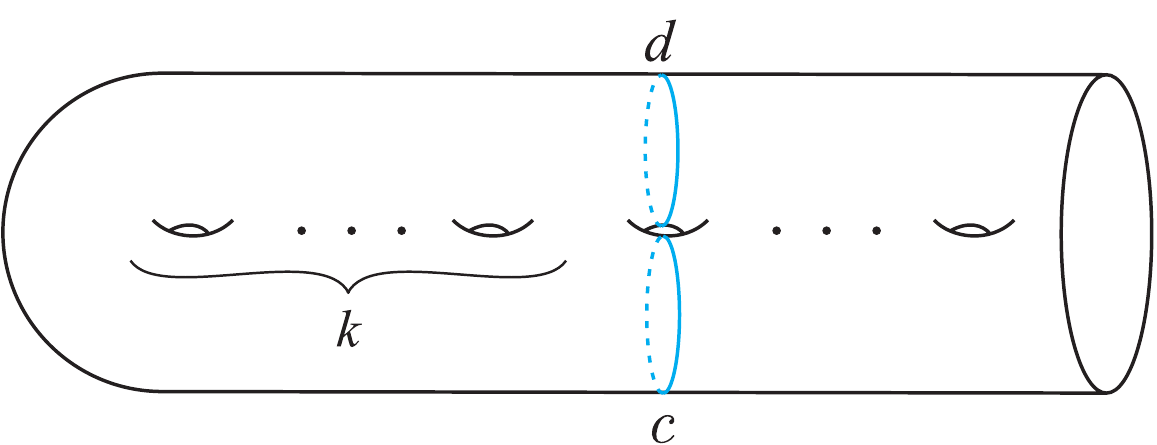}}
\label{figure:bp}
\caption{A genus $k$ bounding pair.}
\end{figure}

\item \textbf{Separating Twists.}  A simple closed curve $c$ is called \emph{separating} if $S-c$ is not connected.  A \emph{separating twist} is  a Dehn twist about a separating curve. If $S = S_{g,1}$, then we say a separating twist has \emph{genus k} if the subsurface whose boundary is $c$ has genus $k$.  As a side note, when $g = 2$, $\I(S_2) = \K(S_2)$, the subgroup generated by separating twists, as there are no BP-maps in $\I(S_2)$ \cite{mess}.

\begin{figure}[htb]
\centerline{\includegraphics[scale=.4]{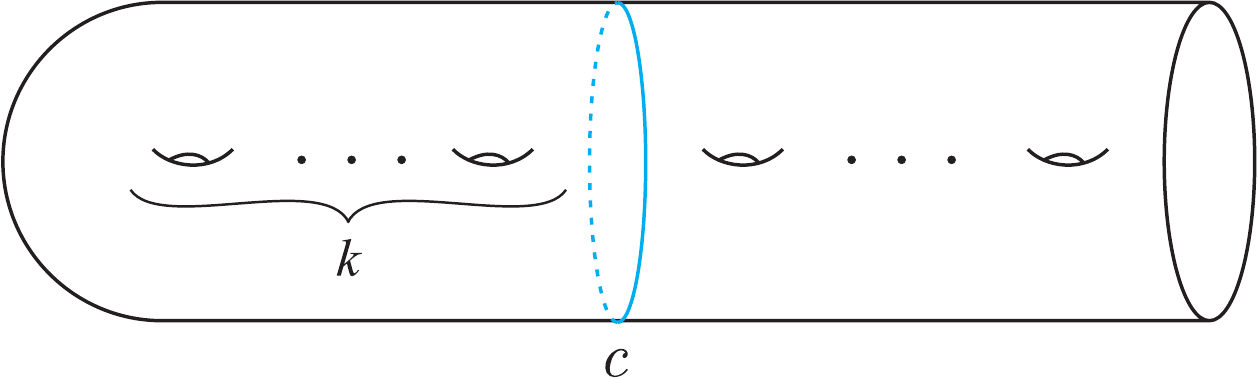}}
\label{figure:sep}
\caption{A genus $k$ separating curve.}
\end{figure}

\item \textbf{Simply Intersecting Pair Maps.} Let $c$ and $d$ be simple closed curves so that $\hat{i}(c,d) = 0$ and $i(c,d) = 2$.  Then a \emph{simply intersecting pair map} (SIP-map) is the commutator of the Dehn twists about the two curves, that is $[T_c,T_d]=T_cT_dT_c^{-1}T_d^{-1}.$\\

\begin{figure}[htb]
\centerline{\includegraphics[scale=.4]{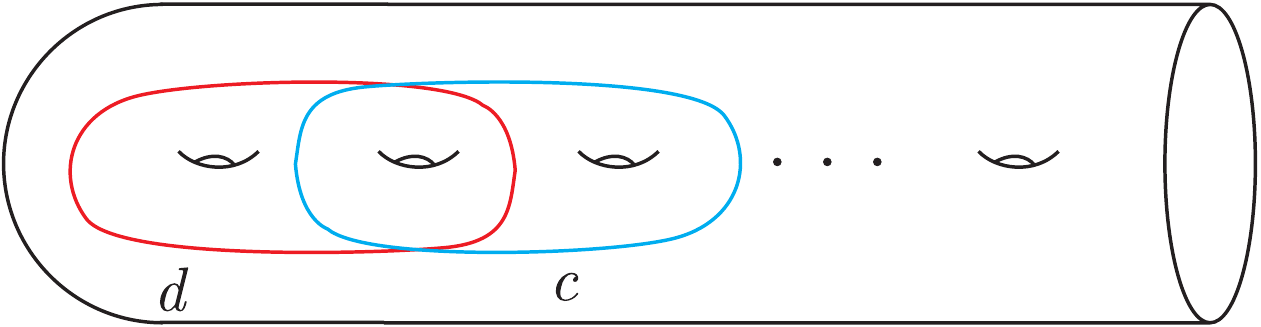}}
\label{figure:sip}
\caption{Simple closed curves $c$ and $d$ that form an SIP-map.}
\end{figure}

\end{enumerate}

In Chapter 7 of \cite{primer}, Farb-Margalit outline how a Dehn twist acts on the homology of a surface.  Let $a$ and $b$ be oriented curves on a surface $S$.  Then
$$[T^k_b(a)] = [a] + k \cdot \hat{i}(a,b)[b].$$
Using this and Lemma~\ref{lem:conjugationrelation}, it is straightforward to show SIP-maps are in $\I(S)$ since 
$$[T_c,T_d]=T_c(T_dT_c^{-1}T_d^{-1})=T_cT_{T_d(c)}^{-1}$$ 
The curves $c$ and $T_d(c)$ are homologous because $[T_d(c)] = [c] + \hat{i}(c,d)[d]=[c]$.  Since twists about homologous curves have the same image under the symplectic representation (see Chapter 7 of \cite{primer} for further details), we can conclude that $T_cT_{T_d(c)}^{-1} \in \I(S)$.  Observe that in essence, SIP-maps are a natural generalization of BP-maps in that they are "inverse products" of Dehn twists about homologous curves and could further be generalized by considering commutators of Dehn twists about curves with higher geometric intersection number which still have algebraic intersection number 0.

While the first two types of elements have been the focus of the literature on $\I(S)$,  SIP-maps have been brought to the forefront due to an infinite presentation of $\I(S_g)$ introduced by Putman that uses all three types of elements \cite{putmaninfinite}.

\section{Basic Facts About SIP-maps}
\label{section:basicfacts}
In this section we will further investigate the structure of SIP-maps.  We begin by showing they are pseudo-Anosov elements on a lantern.

\textbf{Classification of Mapping Classes. }
Mapping classes are often classified according to whether or not they fix any curves in the surface, as follows.
A curve, $c$, is called a {\it reducing curve} for a mapping class $f$, if $f^n(c)=c$ for some $n$.

\textbf{Nielsen-Thurston Trichotomy.}
We are able to classify any mapping class, $f$, into one of the following categories:
\begin{enumerate}
\item The mapping class, $f$, is a \emph{finite order} element; that is, there exists an $n$ such that $f^n=id$
\item  The mapping class, $f$, is \emph{reducible}; that is it fixed a collection of pairwise disjoint curves, or
\item The mapping class, $f$, is \emph{pseudo-Anosov} if it is not finite order or reducible.
\end{enumerate}

There is an equivalent, somewhat more standard and more technical, definition of a pseudo-Anosov mapping class given in terms of measured foliations.  We will not need to use this definition or the machinery of measured foliations explicitly in this work.

There is non-trivial overlap between the finite order and reducible elements.  In order to make this a true trichotomy, we can replace the condition of having a reducing curve with that of having an \emph{essential reducing curve}: a reducing curve $c$ is {\it essential} for a mapping class $h$ if 
for each simple close curve $b$ on the surface such that $i(c,b) \not= 0$, and for each 
integer $m \not= 0$, the classes $h^m(b)$ and $b$ are distinct. 

\begin{thm} [Birman-Lubotzky-McCarthy, \cite{blm}]  
For every mapping class $h$ there exists a system (possibly empty) of essential reducing curves.  Moreover, the system is unique up to isotopy, and there is an $n$ such that
cutting along the system, the restriction of $h^n$ to each component of the cut-open surface is either pseudo-Anosov, finite order, or reducible. 
\end{thm}

Further a fixed curve of a finite order mapping class is never essential, 
because there is always an $n$ such that $h^n=id$ after cutting open along all the
other curves. 

The {\it canonical reduction system} for a mapping class, $f$, is the collection of all essential reducing curves for $f$. This classification, as well as the canonical reduction system, will be used throughout this paper.

Using work of Atalan-Korkmaz we will classify SIP-maps on a lantern, $S_{0,4}$.  They make the following characterizations of reducible elements on the lantern.

\begin{lem}[Atalan-Korkmaz, Lemma 3.4, \cite{atalankorkmaz}]
The reducible elements of $\Mod(S_{0, 4})$ consist of conjugates of nonzero powers of $T_a$, $T_b$ and $T_aT_b$.
\end{lem}

Thus we are able to deduce the following.

\begin{cor}
\label{cor:pseudoanosov}
Let $a$ and $b$ be two curves with $i(a,b) = 2$ and $\hat{i}(a,b)=0$.  Then the SIP-map $f = [T_a,T_b]$ is pseudo-Anosov on a regular neighborhood of $a$ and $b$; that is, on a lantern, $S_{0,4}$.
\end{cor}

\begin{proof}
It is clear $T_aT_bT_a^{-1}T_b^{-1}$ is a  cyclicly reduced word in the free group generated by $T_a$ and $T_b$.  Thus $f$ is not conjugate to a power of $T_a$, $T_b$ or $T_aT_b$ and must be pseudo-Anosov.
\end{proof}

Further, we consider how many SIP-maps are supported on a given lantern.

\begin{prop}
Consider the curves $x, y,$ and $z$ as in Figure~\ref{figure:lantern}.  Then the SIP-maps, $[T_x,T_y], [T_y,T_z],$ and $[T_x,T_z]$, are all distinct, as well as their inverses.
\end{prop}

\begin{proof}
Consider the lantern in Figure~\ref{figure:lantern} with boundary components $a,b,c,$ and $d$ and the lantern relation: $T_xT_yT_z = T_aT_bT_cT_d.$  We consider the SIP-maps $[T_x, T_y]$ and $[T_y, T_z]$.  Suppose $[T_x, T_y] = [T_y, T_z]$.  Using the lantern relation and Lemmas~\ref{lem:conjugationrelation} and \ref{lem:isotopyrelation} we have:
\begin{eqnarray}
[T_x,T_y] & = & [T_y,T_z] \nonumber\\
\iff T_xT_yT^{-1}_xT^{-1}_y & = & T_yT_zT^{-1}_yT^{-1}_z \nonumber\\
\iff T_aT_bT_cT_dT^{-1}_zT^{-1}_xT^{-1}_y & = &  T_aT_bT_cT_dT^{-1}_xT^{-1}_yT^{-1}_z\nonumber\\
\iff T_xT^{-1}_zT^{-1}_x & = & T^{-1}_yT^{-1}_zT_y \nonumber\\
\iff T^{-1}_{T_x(z)} & = & T^{-1}_{T^{-1}_y(z)} \nonumber\\
\iff T_x(z) & = & T^{-1}_y(z)\nonumber
\end{eqnarray}
A simple calculation shows these are not the same curve. Thus $[T_x, T_y] \neq [T_y, T_z]$.  Similar arguments show that the remaining SIP-maps are also distinct.
\end{proof}

Note that distinct pairs of curves can define the same SIP-map.  For example, consider the SIP-maps $[T_z, T_x]$ and $[T_{T^{-1}_y(x)}, T_y]$ where $x, y,$ and $z$ are as in Figure~\ref{figure:lantern}.  Using the lantern relation, Lemma~\ref{lem:conjugationrelation}, and the fact that Dehn twists about disjoint curves commute, we see
\begin{eqnarray}
[T_{T^{-1}_y(x)}, T_y] & = & T_{T^{-1}_y(x)} T_y T^{-1}_{T^{-1}_y(x)} T^{-1}_y \nonumber\\
& = & T^{-1}_yT_xT_yT_yT^{-1}_yT^{-1}_xT_yT^{-1}_y \nonumber\\
& = & T^{-1}_yT_xT_yT^{-1}_x \nonumber\\
& = & (T^{-1}_aT^{-1}_bT^{-1}_cT^{-1}_d T_zT_x)T_x(T_aT_bT_cT_dT^{-1}_xT^{-1}_z)T^{-1}_x\nonumber\\
& = & T_zT_xT^{-1}_zT^{-1}_x\nonumber\\
& = & [T_z,T_x] \nonumber
\end{eqnarray}
In the next section instead of looking at individual SIP-maps, we will look at the group generated by SIP-maps and compare it to well known subgroups of $\I(S)$.

\section{The $\SIP(S_g)$-group}
\label{section:sipgroup}
The goal of this section is to prove some basic results about the group generated by all SIP-maps in $\Mod(S)$, which we will denote as $\SIP(S)$.  We will do this by looking at the image of $\SIP(S)$ under well-known representations of $\I(S)$ as well as classifying which SIP-maps are in the kernel of these representatives.  Recall that we do not distinguish between a curve and its isotopy class.  Similarly, we will frequently not distinguish between a curve and its homology class.  There is an issue regarding the orientation of a curve, and we will deal with this issue when necessary.

\textbf{Johnson Homomorphism. }
Johnson defined a surjective homomorphism, $\tau : \I(S_{g,1})  \longrightarrow \wedge^3 H$ in \cite{johnsonabelian} that measures the action of $f \in \I(S_{g,1})$ on $\pi_1(S).$  Johnson showed that separating twists are in $\text{ker }\tau$.  Further he showed separating twists generate ker $\tau$. We call this subgroup the \emph{Johnson kernel}, $\K(S)$. 

\begin{thm}[Johnson, \cite{johnsonabelian} and \cite{johnson2}]
\label{thm:johnsonkernel}
The group \text{ker} $\tau$ is generated by Dehn twists about separating curves.
\end{thm}

In addition, Johnson showed how to calculate the image of a BP-map under $\tau$ by first choosing a symplectic basis $\{a_1, \dots, a_k, b_1, \dots, b_k \}$ for the homology of the subsurface bounded by $c$ and $d$.  With the chosen basis he showed
$$ \tau(T_cT_d^{-1})=\sum_{i=1}^k (a_i \wedge b_i) \wedge c$$
Note that the orientation of $c$ is chosen so that the subsurface not containing the boundary component is on the left.  Johnson also showed that the image is independent of the choice of symplectic basis. For our purposes we will take this as the definition of $\tau$, since BP-maps generate $\I(S_{g,1})$ when $g \geq 3$. We will usually use the standard symplectic basis for $H$ shown in Figure~\ref{figure:homologybasis}.

\begin{figure}[h]
\centerline{\includegraphics[scale=.7]{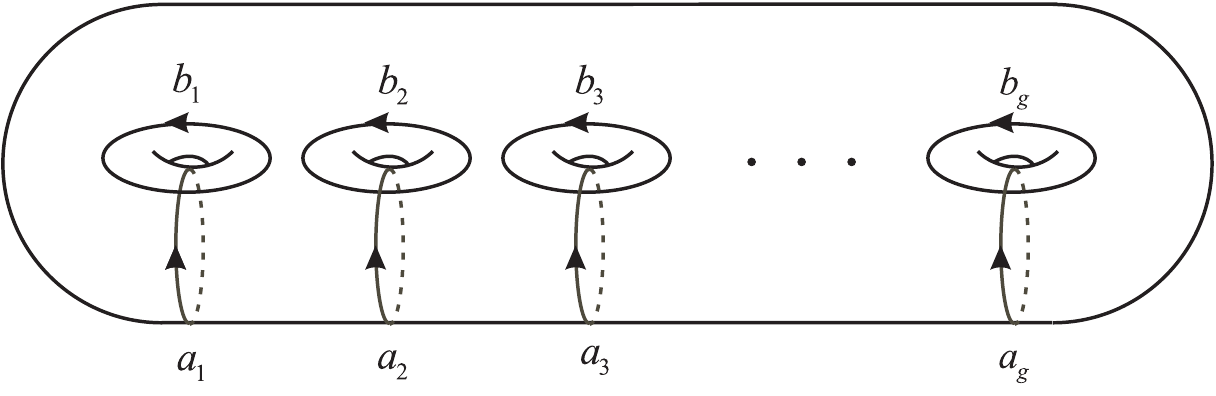}}
\caption{A collection of oriented curves that form a symplectic homology basis for $H_1(S,\Z)$.}
\label{figure:homologybasis}
\end{figure}

It is natural to ask what the image of an SIP-map is under $\tau$.  One way to calculate this is by factoring the SIP-map into BP-maps.  Consider the SIP-map $[T_a,T_b]$ as shown in Figure~\ref{figure:factorsip1}. 

\begin{figure}[htb]
\centerline{\includegraphics[scale=.6]{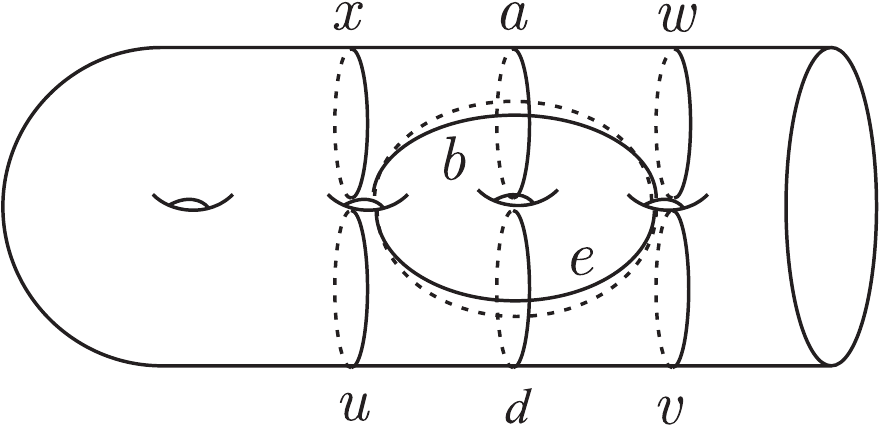} \hspace{2cm}
\includegraphics[scale=.6]{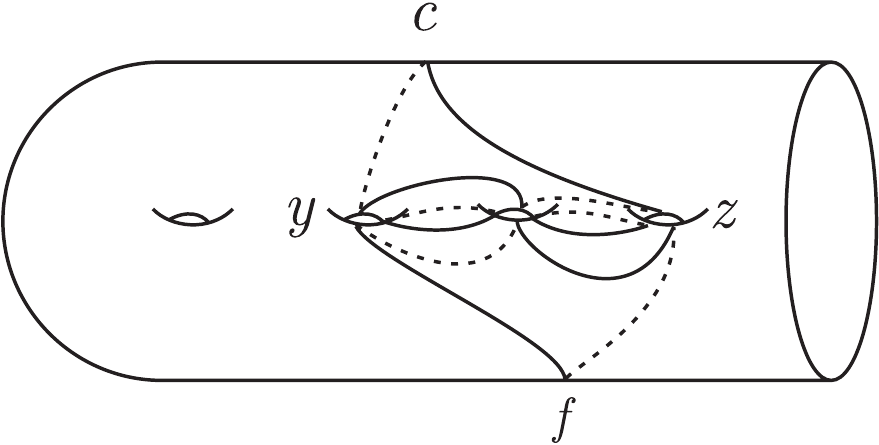}
}
\caption{A collection of curves needed to rewrite the SIP-map, $[T_a, T_b]$ in terms of BP-maps.}
\label{figure:factorsip1}
\end{figure}

\begin{lem}
\label{lem:factorsip}
Let the commutator $[T_a,T_b]$ be an SIP-map as shown in Figure~\ref{figure:factorsip1}.  Then $[T_a,T_b]$ can be rewritten as shown. $$[T_a,T_b] = (T_xT_u^{-1})(T_wT_v^{-1})(T_fT_c^{-1})(T_dT_a^{-1})(T_eT_b^{-1})$$
\end{lem}
\begin{proof}
We will need to use the lantern relation twice to rewrite this SIP-map in terms of BP-maps.\\ \\

\textbf{From the Top Lantern: } $T_aT_bT_c = T_xT_yT_zT_w$\\
\textbf{From the Bottom Lantern: } $T_fT_dT_e=T_yT_zT_vT_u$\\ \\
Then using the above facts and disjointness, we see
\begin{eqnarray}
[T_a,T_b] & = & (T_aT_b)T_a^{-1}T_b^{-1} \nonumber\\
& = & T_x(T_yT_z)T_wT_c^{-1}T_a^{-1}T_b^{-1} \nonumber\\
& = & T_xT_v^{-1}T_u^{-1}T_fT_dT_eT_wT_c^{-1}T_a^{-1}T_b^{-1} \nonumber\\
& = & (T_xT_u^{-1})(T_wT_v^{-1})(T_fT_c^{-1})(T_dT_a^{-1})(T_eT_b^{-1}) \nonumber
\end{eqnarray}  
\end{proof}
Now we are ready to compute the image of an SIP-map under $\tau$. In \cite{johnsonabelian} Johnson gave topological formulas for the image under $\tau$ of BP-maps and separating twists.  Similarly we will give a topological formula for the image of an SIP-map.  We will rely on a common principle used in the study of mapping class groups called the \emph{change of coordinates principle}.  The idea is that to prove a topological statement about a certain configuration of curves, if suffices to show the result on our ``favorite" example of curves satisfying the condition.  For example to show a result about a non-separating curve, up to homeomorphism, it suffices to show the result for any non-separating curve.  See Section 1.3 of \cite{primer} for further details.  We will make use of this principle in proving many of our main results.
\begin{prop}
\label{prop:johnsonsip}
The image under $\tau$ of an SIP-map whose associated lantern has boundary components $w, x, y,$ and $z$ is given by $\tau(f) = \pm [x] \wedge [y] \wedge [z]$.
\end{prop}
\begin{proof}
Let $f = [T_a, T_b]$. Then by the change of coordinates principle, showing the result for $f$ will suffice to prove the general result.
\begin{eqnarray}
\tau([T_a,T_b]) & = & \tau((T_xT_u^{-1})(T_wT_v^{-1})(T_fT_c^{-1})(T_dT_a^{-1})(T_eT_b^{-1})) \nonumber \\
& = & \tau(T_xT_u^{-1}) + \tau(T_wT_v^{-1}) + \tau(T_fT_c^{-1}) + \tau(T_dT_a^{-1}) + \tau(T_eT_b^{-1}) \nonumber\\
& = &  (a_1 \wedge b_1) \wedge [x] +  (a_1 \wedge b_1 + a_2 \wedge b_2 + a_3 \wedge b_3) \wedge [w]  + \nonumber\\
& & (a_1 \wedge b_1 + a_2 \wedge (b_2 - a_3 + b_3)) \wedge [f] +  (a_1 \wedge b_1 + a_2 \wedge b_2) \wedge [d] +  \nonumber\\
& & ( (-a_2 + a_3) \wedge b_3) \wedge [e] \nonumber\\
& = &  (a_1 \wedge b_1) \wedge (-a_2) +  (a_1 \wedge b_1 + a_2 \wedge b_2 + a_3 \wedge b_3) \wedge (-a_4)  + \nonumber\\
& & (a_1 \wedge b_1 + a_2 \wedge (b_2 - a_3 + b_3)) \wedge (a_2 - a_3 + a_4) + \nonumber\\ 
& & (a_1 \wedge b_1 + a_2 \wedge b_2) \wedge (a_3) +  ( (-a_2 + a_3) \wedge b_3) \wedge (-a_2 + a_4) \nonumber\\
& = & - a_2 \wedge a_3 \wedge a_4 \nonumber\\
& = & \pm [x] \wedge [y] \wedge [z] \nonumber
\end{eqnarray}
Observe that every SIP-map is naturally embedded in a lantern with boundary components $w, x, y,$ and $z$, hence we see the image of an SIP-map is $\tau([T_a, T_b]) = \pm x \wedge y \wedge z$ where the orientations of $w$, $x$, $y$ and $z$ are so that the lantern is on the left.  The sign is dependent on the ordering of the boundary components with respect to $a$ and $b$.  
\end{proof}

Recently Putman \cite{putmanjker} and independently Church \cite{church} also calculated the image of a SIP-map under $\tau$ directly, that is without using the above factorization.

\begin{thm} 
\label{thm:sipproper}
The subgroup $\SIP(S_{g,1})$, is a proper subgroup of $\I(S_{g,1})$.
\end{thm}

\begin{proof}
To show this we will make use of the contraction map $\C$ which Johnson introduces in \cite{johnsonabelian}. The contraction map $\C : \wedge^3 H \longrightarrow H$ is defined by 
$$a \wedge b \wedge c \longmapsto 2( \ia(b,c)a + \ia(a,c)b + \ia(a,b)c).$$  
Hence using Proposition~\ref{prop:johnsonsip} it is easy to see that SIP-maps are in the kernel of $(\C \circ \tau)$ since the boundary components of a lantern are disjoint.
$$(\C \circ \tau)([T_c, T_d]) = \C(\pm w \wedge x \wedge y) = 0.$$
 Further, Johnson shows that $\C \circ \tau$ actually maps $\I(S_{g,1})$ onto $2H$.  From this, we are able to deduce that $\I(S_{g,1}) \neq \SIP(S_{g,1})$.  
 \end{proof}
The following corollaries are immediate consequences of the proof of Theorem~\ref{thm:sipproper} and Proposition~\ref{prop:johnsonsip}.
\begin{cor}
The group $\SIP(S_{g,1}) \nsubseteq \K(S_{g,1})$.
\end{cor}

\begin{cor}
\label{cor:infiniteindex}
The group, $\SIP(S_{g,1})$, is an infinite index subgroup of $\I(S_{g,1})$.
\end{cor}

\textbf{SIP-maps in $\K(S)$.}
We can now characterize which SIP-maps are in $\K(S) = \textrm{ker } \tau$.

\begin{cor}
\label{cor:sipkernel}
An SIP-map $f$ is an element of $\K(S)$ if and only if the lantern associated with $f$ has a boundary component that is null-homologous or if two boundary components are homologous. 
\end{cor}

\begin{proof}
This follows directly from the calculation given in Proposition~\ref{prop:johnsonsip} of $\tau (f) = \pm [x] \wedge [y] \wedge [z]$. 
\end{proof}

See Figure~\ref{figure:specialsip} for examples of each type of SIP-map in $\K(S)$.

\begin{figure}[htb]
\centerline{\includegraphics[scale=.5]{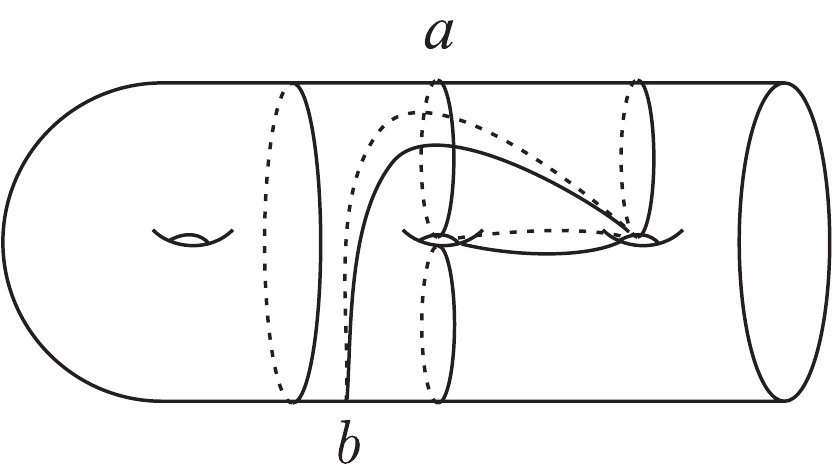}  \hspace{1 cm}   \includegraphics[scale=.5]{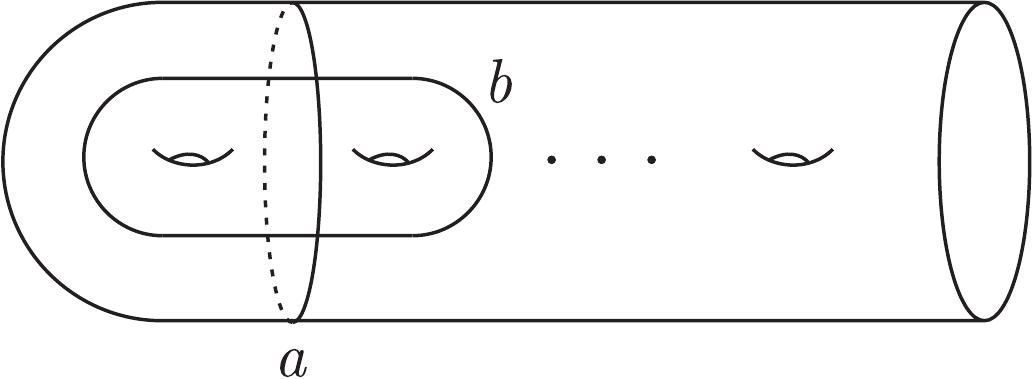}}
\caption{Examples of SIP-maps in $\K(S)$.}
\label{figure:specialsip}
\end{figure}

\textbf{The Subgroup $\langle \SIP(S_{g,1}) , \K(S_{g,1}) \rangle$. } The group $\langle \SIP(S_{g,1}) , \K(S_{g,1}) \rangle$ has appeared in the literature before this, but has never been recognized in terms of SIP-maps.  We will define basic terminology regarding \emph{winding numbers} and the \emph{Chillingworth subgroup}. Then we show how $\langle \SIP(S_{g,1}) , \K(S_{g,1}) \rangle$ can be viewed in four different ways.

More can be said about the structure of $\wedge^3H$, and it can be applied to our current situation.  According to Sakasai, $\wedge^3H$ has two irreducible components as Sp-modules (Section 2.3, \cite{sakasai}):
$$\wedge^3H = H \oplus U$$
where $U$ is the kernel of $\C$, the contraction map. It follows from irreducibility and normality that $\tau(\SIP(S_{g,1})) = U$.  From the following commutative diagram we see that $$\SIP(S_{g,1}) \diagup (\K(S_{g,1}) \cap \SIP(S_{g,1})) \cong U.$$

\begin{equation*}
  \xymatrix@C+1em@R+1em{
   \SIP(S_{g,1}) \ar[r]^-{\tau} \ar[d]_-{} & \text{ker }\C = U \\
   \SIP(S_{g,1})/(\K(S_{g,1}) \cap \SIP(S_{g,1})) \ar@{-->}[ur]_-{}
  }
 \end{equation*}

Further, ker $(\C \circ \tau) = \langle \SIP(S_{g,1}) , \K(S_{g,1}) \rangle$.  From this we can conclude 
$$\I(S_{g,1}) \diagup \langle \SIP(S_{g,1}) , \K(S_{g,1} \rangle) \cong 2H.$$
Because $2H$ is an infinite group and $\SIP(S_{g,1}) \subset \langle \SIP(S_{g,1}) , \K(S_{g,1}) \rangle$, it also follows that $\SIP(S_{g,1})$ is of infinite index in $\I(S_{g,1})$.

\textbf{Winding Number.} For a surface $S$ with continuous, non-vanishing vector field $X$ on $S$, Chillingworth defines the concept of \emph{winding number} with respect to $X$ of an oriented regular curve, $c$, to be the number of times its tangent rotates with respect to the framing induced by $X$ \cite{chillingworth1,chillingworth2}, denoted as $\omega_X(c)$.  If $f \in \I(S)$, we have the function $$e_{f,X}(c) = \omega_X(f(c)) - \omega_X(c).$$  This function measures the change in winding number induced by $f$.  Johnson showed this function is independent of the choice of vector field $X$ \cite{johnsonabelian}, hence we will write $e_f$.  Note that Johnson also showed $e_f$ is a function on homology classes.  We can then dualize the class $e_f$ to a homology class $t_f$ where $c \cdot t_f = e_f(c)$.  We call $t_f$ the \emph{Chillingworth class} of $f$.  Johnson showed that $t_f = (\C \circ \tau)(f)$.  Thus we have shown $\langle \SIP(S_{g,1}) , \K(S_{g,1}) \rangle$ is the kernel of $t$.  The kernel of $t$ is also called the \emph{Chillingworth subgroup}.

\textbf{Chillingworth Subgroup.} Trapp showed in \cite{trapp} that the Chillingworth subgroup is characterized as:

\begin{eqnarray*}
\{ f = T_{\gamma_1}T^{-1}_{\delta_1}T_{\gamma_2}T^{-1}_{\delta_2} \dots T_{\gamma_n}T^{-1}_{\delta_n} & | & T_{\gamma_i}T^{-1}_{\delta_i} \textrm{ is a genus one BP-map} \\
& & \textrm{and } \sum_{i=1}^n 2 [\gamma_i]= 0 \textrm{ in } H \}
\end{eqnarray*}

This follows from a calculation done by Trapp and Johnson that if $T_{\gamma}T^{-1}_{\delta}$ is a genus one BP-map then $t(T_{\gamma}T^{-1}_{\delta}) = 2 [\gamma]$.  Further, we can extend this presentation to include BP-maps of genus $g$ in the following way.  If the BP-map $T_{\gamma}T^{-1}_{\delta}$ has genus $g(\gamma,\delta)$ with $a_i, b_i$ as a symplectic basis for the corresponding subsurface, then 
\begin{eqnarray}
t(T_{\gamma}T^{-1}_{\delta}) & = & \C((\sum_{i=1}^{g(\gamma,\delta)}a_i \wedge b_i)\wedge \gamma) \nonumber\\
& = & \C(\sum_{i=1}^{g(\gamma,\delta)}a_i \wedge b_i \wedge \gamma)\nonumber\\
& = & \sum_{i=1}^{g(\gamma,\delta)}\C(a_i \wedge b_i \wedge \gamma)\nonumber\\
& = &\sum_{i=1}^{g(\gamma,\delta)}2[\gamma]\nonumber\\
& = & 2g(\gamma, \delta)[\gamma]\nonumber
\end{eqnarray}

So we could write the Chillingworth subgroup as:

\begin{eqnarray*}
\{ f = T_{\gamma_1}T^{-1}_{\delta_1}T_{\gamma_2}T^{-1}_{\delta_2} \dots T_{\gamma_n}T^{-1}_{\delta_n} & | & T_{\gamma_i}T^{-1}_{\delta_i} \textrm{ is a BP-map} \\
& & \textrm{and } \sum_{i=1}^n 2 g(\gamma_i,\delta_i) [\gamma_i]= 0 \textrm{ in } H \}
\end{eqnarray*}

Equivalently, we can include separating twists, $T_\gamma$, because $t(T_\gamma)=0$.  Hence the Chillingworth subgroup is:

\begin{eqnarray*}
\{ f = f_1 f_2 \dots f_n  & | & f_i = T_{\gamma_i} \textrm{ and } \gamma_i \textrm{ is a separating curve or}\\
& & f_i = T_{\gamma_i}T^{-1}_{\delta_i} \textrm{ is a  BP-map and } \sum_{i=1}^n 2 g(\gamma_i,\delta_i) [\gamma_i]= 0 \textrm{ in } H,\\
& & \textrm{ with } g(\gamma_i,\delta_i) := 0 \textrm{ if } \gamma_i \textrm{ is separating }\}
\end{eqnarray*}

Now we have the following equivalence:
\begin{cor} 
The following are equivalent definitions of the group $\langle \SIP(S_{g,1}) , \K(S_{g,1}) \rangle$.
\begin{enumerate}
\item The group $\langle \SIP(S_{g,1}) , \K(S_{g,1}) \rangle$ is the group generated by all separating twists and SIP-maps.\\
\item The group $\langle \SIP(S_{g,1}) , \K(S_{g,1}) \rangle$ is the kernel of $\C \circ \tau$.\\
\item The group $\langle \SIP(S_{g,1}) , \K(S_{g,1}) \rangle$ is the group of all elements in $\I(S_{g,1})$ with winding number zero.\\
\item The group $\langle \SIP(S_{g,1}) , \K(S_{g,1}) \rangle$ is the following group:
\begin{eqnarray*}
\{ f = T_{\gamma_1}T^{-1}_{\delta_1}T_{\gamma_2}T^{-1}_{\delta_2} \dots T_{\gamma_n}T^{-1}_{\delta_n} & | & T_{\gamma_i}T^{-1}_{\delta_i} \textrm{ is a genus one BP-map} \\
& & \textrm{and } \sum_{i=1}^n 2 [\gamma_i]= 0 \textrm{ in } H \}
\end{eqnarray*}
\end{enumerate}
\end{cor}

\section{Reinterpreting Relations}
\label{section:relations}
A potential application of studying SIP-maps is to find a better generating set for $\I(S)$.  While Johnson found a finite generating set for $\I(S_g)$ when $g \geq 3$ it is extremely large \cite{johnson1}.  Johnson conjectured that this generating set could be reduced to a more manageable size.  Johnson's main technique was to employ several relations he discovered among BP-maps.  We will use SIP-maps to reinterpret one of Johnson's relations in $\I(S)$.  Perhaps similar techniques could be used to rewrite the other Johnson relations.

Independently, Putman showed how to factor an SIP-map into the product of two BP-maps \cite{putmaninfinite}.  

\begin{lem}[Putman, Fact F.5, \cite{putmaninfinite}]
\label{lem:factorputman}
Let curves $a, b,$ and $e$ be as in Figure~\ref{figure:factorsip1}.  Then 
$$[T_a,T_b] = (T_{T_a(b)}T^{-1}_e)(T_e T^{-1}_b).$$
\end{lem}

Combining Lemmas~\ref{lem:factorsip} and \ref{lem:factorputman} yields a new proof of the following relation in $\I(S)$ discovered by Johnson \cite{johnson1}.

\begin{figure}[htb]
\centerline{\includegraphics[scale=.4]{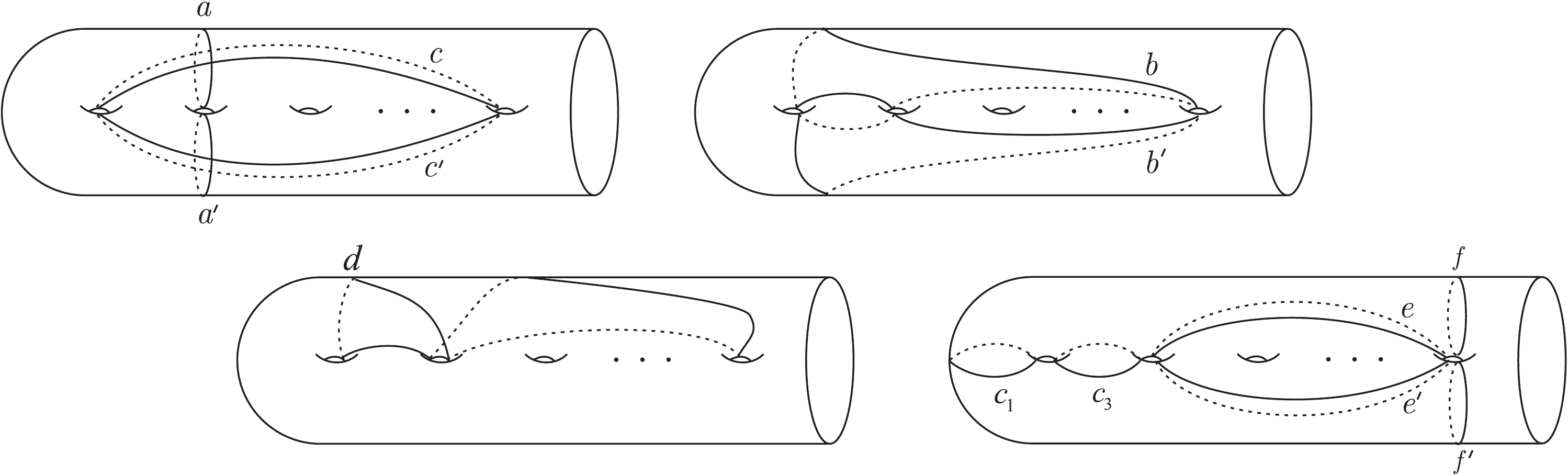}}
\caption{The curves needed for Johnson's relation.}
\label{figure:johnsonlemma10}
\end{figure}

\begin{lem}[Johnson, Lemma 10, \cite{johnson1}]
Let curves $a, a', b, b', c, c', c_1, c_3,$ $d, e, e', f,$ and $f'$ be as defined in Figure~\ref{figure:johnsonlemma10}.  Then 
$$T_a T^{-1}_{a'} T_b T^{-1}_{b'} T_d T^{-1}_{c'} = T_e T^{-1}_{e'} T_f T^{-1}_{f'}.$$  
\end{lem}

\begin{proof}
Using the factoring of Putman in Lemma~\ref{lem:factorputman} we see:
$$[T_a, T_c] = T_d T^{-1}_{c'} T_{c'} T^{-1}_c, \textrm{ where } d = T_a(c).$$
Using the techniques and factoring in Lemma~\ref{lem:factorsip} we have the following:
$$[T_a, T_c] = T_{b'} T^{-1}_{b} T_{a'} T^{-1}_{a} T_{c'} T^{-1}_{c} T_{f} T^{-1}_{f'}.$$
Combining these two equations and rearranging we have the desired result.
\begin{eqnarray}
T_d T^{-1}_{c'} T_{c'} T^{-1}_c & = &  T_{b'} T^{-1}_{b} T_{a'} T^{-1}_{a} T_{c'} T^{-1}_{c} T_{f} T^{-1}_{f'} \nonumber\\
T_a T^{-1}_{a'} T_b T^{-1}_{b'} T_d T^{-1}_{c'} & = & T_e T^{-1}_{e'} T_f T^{-1}_{f'} \nonumber
\end{eqnarray} 
\end{proof}

Further study is needed to determine whether other relations among Johnson's generators in $\I(S)$ can be realized by SIP-maps.  

\section{Birman-Craggs-Johnson Homomorphism}
\label{section:bcj}
In order to define the Birman-Craggs-Johnson homomorphism, one of the most well known representations of the Torelli group, we will need first to consider the Birman-Craggs homomorphisms.  

\textbf{Birman-Craggs Homomorphisms. } In \cite{birmancraggs} Birman and Craggs introduced a finite collection of homomorphisms from $\I(S_{g,1})$ to $\Z_2$ based on the Rochlin invariant. 
 
\textbf{Rochlin Invariant.} Let $W$ be a homology sphere and $X$ be a simply connected parallelizable 4-manifold, so $W=\partial X$.  We know such a manifold $X$ always exists, and the signature$(X)$ is divisible by 8.  Further,  $\frac{\mathrm{signature}(X)}{8}\mathrm{mod} \,  2$ is independent of $X$; hence, an invariant of $W$ called the \emph{Rochlin invariant} denoted by $\mu$.  A good reference for this material is \cite{gompf}.

The \emph{Birman-Craggs homomorphisms} is a collection of homomorphisms $$\rho_h: \I \to\Z_2$$ defined by fixing an embedding $h: S \hookrightarrow S^3$ and identifying $S$ with $h(S).$  For $f \in \I$, split $S^3$ along $S$ and reglue the two pieces using $f,$ creating a closed 3-manifold, $W(h,f)$.  Since $f$ acts trivially on $H_1(S, \Z)$, the 3-manifold $W(h,f)$ is a homology sphere.  Thus the Rochlin invariant $\mu(h,f) \in \Z_2$ is defined.  Hence for a  fixed embedding $h$,  
$$\rho_h(f)  : =  \mu(h,f)$$ 
is the Birman-Craggs homomorphism.  In addition, Johnson showed these homomorphisms correspond to the mod 2 self-linking forms associated with $S$, hence there are only finitely many \cite{johnsonbcj}.

\textbf{Birman-Craggs-Johnson Homomorphism.}  
In \cite{johnsonbcj}, Johnson combined all the Birman-Craggs homomorphisms into one homomorphism $\sigma$, in the sense that the kernel of $\sigma$ is equal to the intersection of the kernels of all the Birman-Craggs homomorphisms.  In order to describe this homomorphism, we first need to define boolean polynomials.

We construct from $H_1(S, \mathbb{Z}_2)$ a $\mathbb{Z}_2$-algebra $B$ such that:
\begin{enumerate}
\item $B$ is commutative with unity
\item $B$ is generated by the abstract elements $\bar{a}$ where $a$ is nonzero in $H_1(S, \mathbb{Z}_2),$
\item $\bar{a}^2=\bar{a}$ for all $a \neq 0$ in $H_1(S, \mathbb{Z}_2)$. (Sometimes this is referred to as a ``square-free" algebra.)
\item $\overline{(a+b)} = \bar{a} + \bar{b} + a\cdot b$ where $a \cdot b \in \mathbb{Z}_2 \subset B$ is the algebraic intersection of $a$ and $b$ modulo 2.
\end{enumerate}

Elements of $B$ are thought of as polynomials in the generators.  The \emph{degree} of an element is well defined, and we let $B_3$ equal the vector space of all elements in $B$ of degree less than or equal to 3.

Then the \emph{Birman-Craggs-Johnson homomorphism} is a surjective homomorphism
$$\sigma:  \I(S_{g,1}) \longrightarrow B_3.$$

Johnson also calculated the image of BP-maps and separating twists under $\sigma$.  Again since BP-maps generate $\I(S_g)$ when $g \geq 3$,  for our purposes we will take this as the definition of the map $\sigma$.  Thus,

\begin{enumerate}
\item A genus $k$ separating curve, $c$:

\begin{itemize}
\item Choose a symplectic basis $\{a_1, \cdots , a_k, b_1, \cdots , b_k\}$ for the subsurface bounded by $c$. 
\item $\sigma(T_c) = \sum_{i=1}^k \bar{a}_i \bar{b}_i $
\end{itemize}

\item A genus $k$ BP-map $T_cT_d^{-1}$:

\begin{itemize}
\item Choose a symplectic basis $\{a_1, \cdots , a_k, b_1, \cdots , b_k\}$ for the subsurface bounded by $c$ and $d$. 
\item $\sigma(T_cT_d^{-1}) = \sum_{i=1}^k \bar{a}_i \bar{b}_i(1-\bar{c}) $
\end{itemize}

\end{enumerate}
Johnson showed both these calculations are independent of the choice of symplectic basis.  We will find a similar formula for the image of an SIP-map under $\sigma$ that is based completely on the topological structure of the given SIP-map.

Given the rewriting of an SIP-map in terms of BP-maps in Lemma~\ref{lem:factorsip}, it is not hard to determine the image of an SIP-map under $\sigma$.

\begin{prop}
\label{prop:sipbcj}
Consider the SIP-map, $[T_a, T_b]$, as shown in Figure~\ref{figure:factorsip1}, which is naturally embedded in a lantern with boundary components $w, x, y,$ and $z$. Then $\sigma([T_a, T_b]) =  \bar{x} \bar{y} \bar{z}$.
\end{prop}

\begin{proof}
The proof consists of the following calculation using Lemma~\ref{lem:factorsip} and the change of coordinates principle.  We will be using the standard symplectic basis show in Figure~\ref{figure:homologybasis} for this calculation.
\begin{eqnarray}
[T_a,T_b]  & =  & (T_xT_u^{-1})(T_wT_v^{-1})(T_fT_c^{-1})(T_dT_a^{-1})(T_eT_b^{-1}) \nonumber\\
& = & \sigma(T_xT_u^{-1}) + \sigma(T_wT_v^{-1}) + \sigma(T_fT_c^{-1}) + \sigma(T_dT_a^{-1}) + \sigma(T_eT_b^{-1}) \nonumber\\
& = & \bar{a}_1\bar{b}_1(1-\bar{a}_2) + (\bar{a}_1\bar{b}_1 + \bar{a}_2\bar{b}_2 + \bar{a}_3\bar{b}_3)(1-\bar{a}_4)\nonumber\\
& & (\bar{a}_1\bar{b}_1 + \bar{a}_2 (\overline{b_2-a_3+b_3}))(1 - (\overline{-a_2 + a_3 - a_4})) \nonumber\\
& & (\bar{a}_1\bar{b}_1 +\bar{a}_2\bar{b}_2)(1+\bar{a}_3) + ((\overline{-a_2+a_3})\bar{b}_3)(1-\overline{a_2 -a_4})\nonumber\\
& = & \bar{a}_2\bar{a}_4 + \bar{a}_2\bar{a}_3\bar{a}_4 \nonumber\\
& = & \bar{a}_2 (\overline{-a_2 + a_4})(\overline{-a_3 + a_4}) \nonumber\\
& = & \bar{x} \bar{y} \bar{z} \nonumber
\end{eqnarray}
Note that since $w, x, y,$ and $z$ bound a subsurface, the result is equivalent to that using any three of the four bounding curves.  For example, we consider $\bar{x} \bar{y} \bar{w}$.
\begin{eqnarray}
\bar{x} \bar{y} \bar{w} & = & \bar{x} \bar{y} (\overline{x+y+z}) \nonumber\\
& = & \bar{x}\bar{y}(\bar{x} + \bar{y} + \bar{z}) \nonumber\\
& =& \bar{x}\bar{y} +\bar{x}\bar{y}+\bar{x}\bar{y}\bar{z}\nonumber\\
& = & \bar{x}\bar{y}\bar{z} \nonumber
\end{eqnarray}
\end{proof}

\begin{cor}
\label{cor:sipbcj}
An SIP-map, $[T_a, T_b]$, with associated lantern, $L$, is an element of $\textrm{ker }\sigma$ if and only if one of the boundary components of $L$ is null-homologous.
\end{cor}

Observe that an SIP-map, $[T_a,T_b]$, where $a$ or $b$ is a separating curve, is always in $\K(S)$ and sometimes in the kernel of $\sigma.$  We call these \emph{separating SIP-maps}.  Let SSIP$(S)$ be subgroup generated by separating SIP-maps.  We will compute the image of SSIP$(S)$ under $\sigma$ and deduce that SSIP$(S)$ is a proper subgroup of $\K(S)$.

\begin{thm}
\label{thm:ssip}
The image of the subgroup generated by separating SIP-maps, that is SSIP$(S_{g,1})$, under the Birman-Craggs-Johnson homomorphism, $\sigma$, is $\langle 1, \bar{a}_i,   \bar{b}_i,   \bar{a}_i \bar{b}_j,   \bar{a}_i\bar{b}_i+\bar{a}_j \bar{b}_j  |  1 \leqslant i, j \leqslant g, i \ne j \rangle$.  
\end{thm}
An immediate consequence of this result and Theorem~\ref{thm:johnsonkernel} is the following:
\begin{cor}
Let $S$ be a surface with genus, $g \geq 3$. Then the group generated by separating SIP-maps, SSIP$(S_{g,1})$, is a proper subgroup of $\K(S_{g,1})$.
\end{cor}
\begin{proof}[Proof of Theorem~\ref{thm:ssip}:]
There are four basic types of generators of $B_2$:
\begin{enumerate}
	\item $\bar{a}_i \bar{b}_i$ 
	\item $\bar{a}_i \bar{b}_j$ with $i \neq j$ (this also includes $\bar{a}_i \bar{a}_j, \bar{a}_i\bar{b}_j,$ and $\bar{b}_i \bar{b}_j$ )
	\item $\bar{a}_i$ ($\bar{b}_i$)
	\item $1$
\end{enumerate}
Elements of type (1) we will refer to as $\emph{symplectic terms}$ and elements of type (2) will be $\emph{nonsymplectic}$. We will show that elements of type (2), (3) and (4) are in $\sigma(\SIP(S) \cap \K(S))$.  Then we will show which elements generated by type (1) terms are in the image.   Recall all coefficients in $B_2$ are in $\mathbb{Z}/2\mathbb{Z}$.

Note that we will only consider separating SIP-maps, that is, SIP-maps where at least one of the defining curves is separating as in Figure~\ref{type1a}.

\textbf{Type (2):}
\begin{figure}[h]
\begin{center}
\includegraphics[scale = .5]{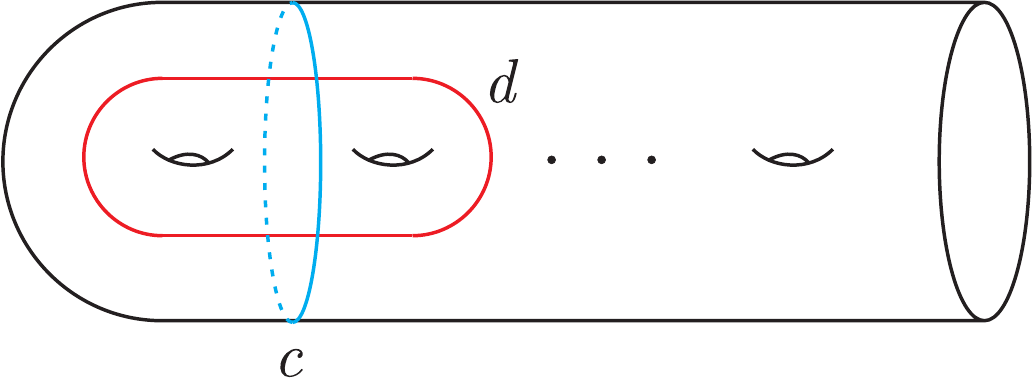}
\caption{An SIP-map, $[T_c,T_d]$ with $c$ a separating curve.}
\label{type1a}
\end{center}
\end{figure}
Suppose $c$ and $d$ are as shown.  Then
\begin{eqnarray} 
\sigma([T_c,T_d]) & = & \sigma(T_cT_dT_c^{-1}T_d^{-1}) \nonumber \\
& = & \sigma(T_c)+\sigma(T_{T_d(c)}^{-1}) \nonumber \\
& = & \sigma(T_c)+\sigma(T_{T_d(c)}) \nonumber 
\end{eqnarray} 
It is not hard to see that 
$$\sigma(T_c) = \bar{a}_1\bar{b}_1.$$
Further $T_d(c)$ is shown below with symplectic basis consisting of $a_1+b_2,$ and $b_1$ (where $a_i$ and $b_i$ are from the standard symplectic basis as shown in Figure~\ref{figure:homologybasis}.  \\

\begin{figure}[h]
\begin{center}
\includegraphics[scale = .5]{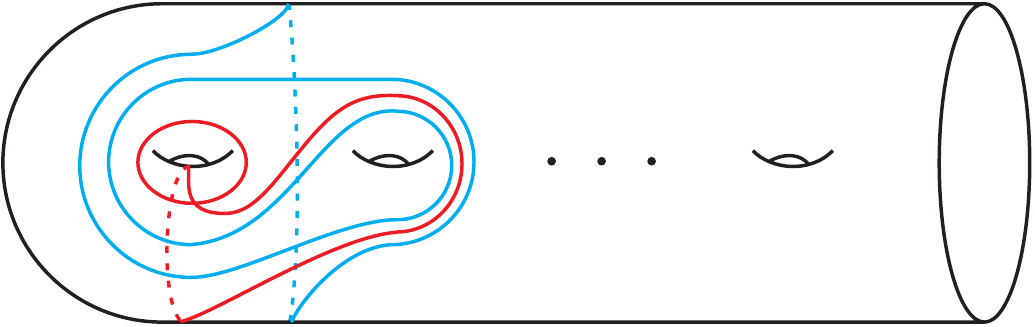}
\caption{$T_d(c)$ with symplectic basis $a_1+b_2,$ and $b_1$}
\label{type1b}
\end{center}
\end{figure}
Hence we see that $\sigma(T_{T_d(c)})  =  (\overline{a_1+b_2}) \bar{b}_1 =  \bar{a}_1 \bar{b}_1 + \bar{b}_1 \bar{b}_2$. Therefore
$\sigma([T_c,T_d]) =  \bar{b}_1 \bar{b}_2.$
By change of coordinates we can get all elements of Type (2).

\textbf{Type (3):}
\begin{figure}[h]
\begin{center}
\includegraphics[scale = .7]{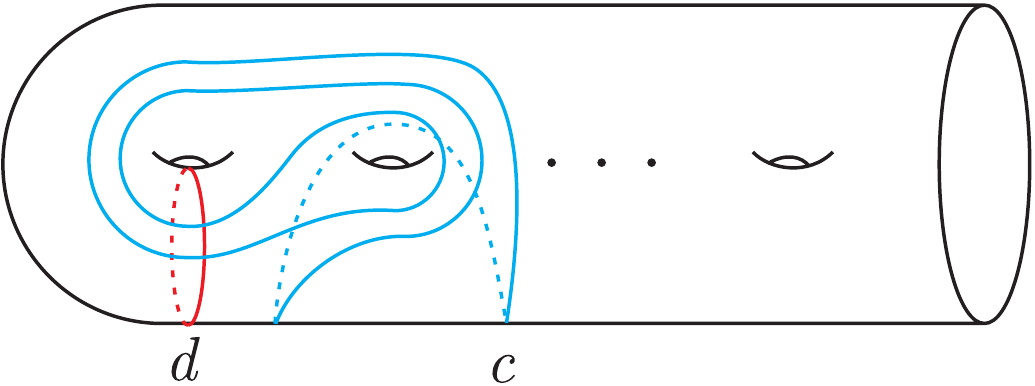}
\caption{SIP-map for Type (3)}
\label{type3a}
\end{center}
\end{figure}
Suppose $c$ and $d$ are as shown in Figure~\ref{type3a}. We want to find $\sigma([T_c,T_d])= \sigma(T_c)+\sigma(T_{T_d(c)}) .$  As shown in Figure~\ref{type3b}, $c$ has symplectic basis $b_2$ and $b_1+a_2+b_2$.

\begin{figure}[h]
\begin{center}
\includegraphics[scale = .7]{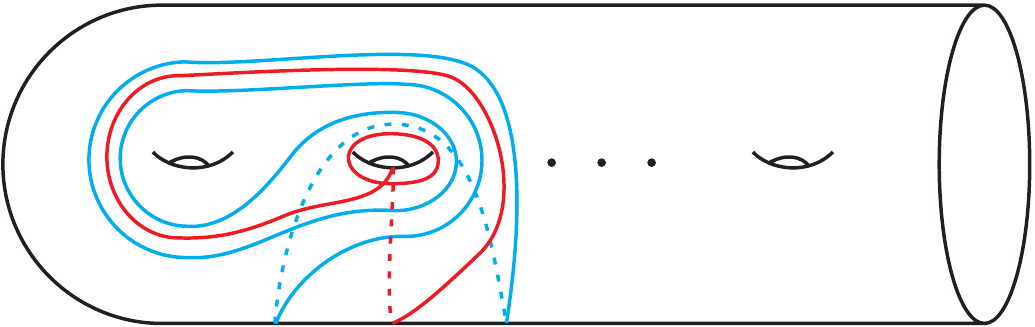}
\caption{Curve $c$ with symplectic basis $b_2$ and $b_1+a_2+b_2$.}
\label{type3b}
\end{center}
\end{figure}

Thus $$\sigma(T_c) = \bar{b}_2(\overline{b_1+a_2+b_2}) = \bar{b}_2(\bar{b}_1 + \bar{a}_2+\bar{b}_2 + 1) = \bar{b}_1  \bar{b}_2+\bar{a}_2\bar{b}_2 $$  Further $T_d(c)$ is shown in Figure~\ref{type3c} with symplectic basis consisting of $b_2$ and $a_1+b_1+a_2+b_2$. 

\begin{figure}[ht]
\begin{center}
\includegraphics[scale = .7]{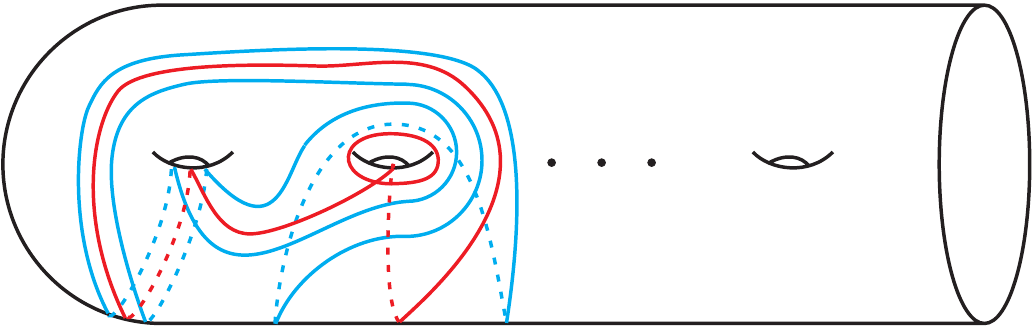}
\caption{Curve $T_d(c)$ with symplectic basis $b_2$ and $a_1+b_1+a_2+b_2$.}
\label{type3c}
\end{center}
\end{figure}

Hence we see that
$$\sigma(T_{T_d(c)})  =  (\bar{b}_2) (\overline{a_1+b_1+a_2+b_2}) =  \bar{a}_1\bar{b}_2+\bar{b}_1\bar{b}_2+\bar{a}_2\bar{b}_2+\bar{b}_2.$$

Therefore
$$\sigma([T_c,T_d]) = \bar{a}_1\bar{b}_2+\bar{b}_2.$$
Since we can get all Type (2) elements by themselves, composing with the appropriate SIP-maps and using change of coordinates we can get all elements of Type (3).

\textbf{Type (4):}
\begin{figure}[h]
\begin{center}
\includegraphics[scale = .7]{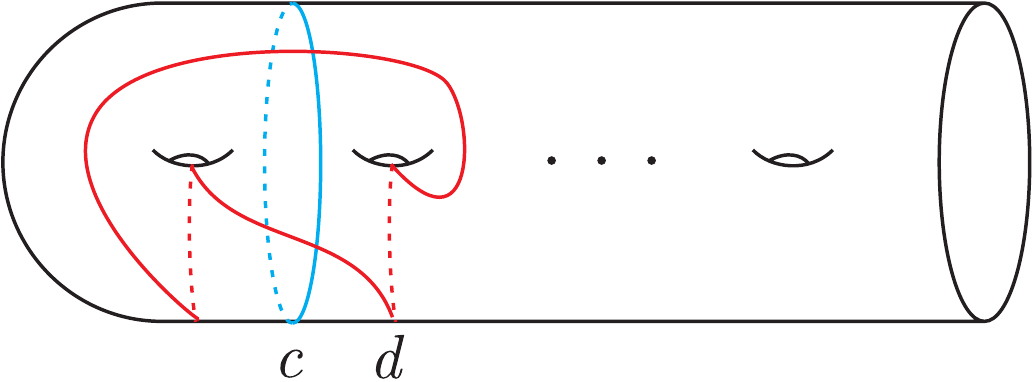}
\caption{SIP-map for Type (4).}
\label{type4a}
\end{center}
\end{figure}
Suppose $c$ and $d$ are as shown in Figure~\ref{type4a}. We want to find $\sigma([T_c,T_d])= \sigma(T_c)+\sigma(T_{T_d(c)}) .$  
Clearly $c$ has symplectic basis $a_1$ and $b_1$, hence $\sigma(T_c) = \bar{a}_1\bar{b}_1.$  Further $T_d(c)$ is shown in Figure~\ref{type4b} with symplectic basis consisting of $a_1+a_2+b_2$ and $b_1+a_2+b_2$. 

\begin{figure}[h]
\begin{center}
\includegraphics[scale = .7]{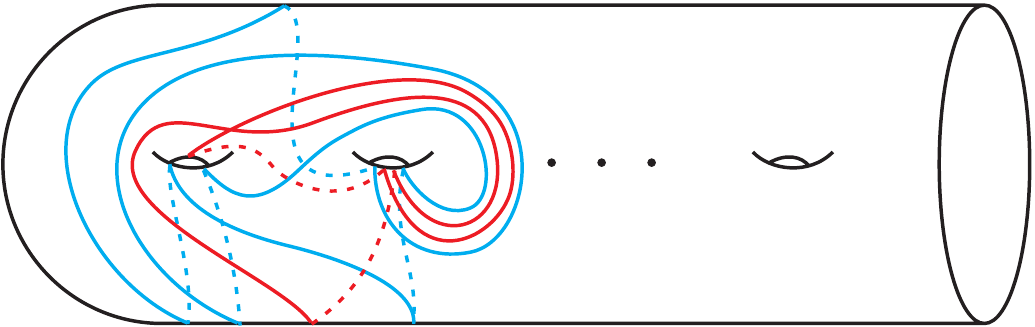}
\caption{Curve $T_d(c)$ with symplectic basis $a_1+a_2+b_2$ and $b_1+a_2+b_2$.}
\label{type4b}
\end{center}
\end{figure}
Hence
\begin{eqnarray}
\sigma(T_{T_d(c)}) & = & (\overline{a_1+a_2+b_2}) (\overline{b_1+a_2+b_2}) \nonumber \\ 
& = & (\bar{a}_1+\bar{a}_2+\bar{b}_2+1)(\bar{b}_1+\bar{a}_2+\bar{b}_2+1) \nonumber \\
& = & \bar{a}_1\bar{b}_1+\bar{a}_1\bar{a}_2+\bar{a}_1\bar{b}_2+\bar{a}_2\bar{b}_1+\bar{b}_1\bar{b}_2+\bar{b}_1+\bar{a}_2+\bar{b}_2+1\nonumber
\end{eqnarray}
Therefore
$$ \sigma([T_c,T_d]) =  \bar{a}_1\bar{a}_2+\bar{a}_1\bar{b}_2+\bar{a}_2\bar{b}_1+\bar{b}_1\bar{b}_2+\bar{b}_1+\bar{a}_2+\bar{b}_2+1.$$
By change of coordinates and because we can get Type (2) and Type (3) elements we can compose by the appropriate SIP-maps to get all Type (4) elements.

\textbf{Type (1): }
Now let us consider which elements of Type (1) are in $\sigma(\text{SSIP}(S))$. Let $c$ be any separating curve.  $\sigma(T_c)$ must have a symplectic term (that is, a term of the form $\bar{a}_i \bar{b}_i$) and possibly other terms.  We know
\begin{eqnarray}
\sigma([T_c,T_d]) & = & \sigma(T_cT_dT_c^{-1}T_d^{-1}) \nonumber \\
& = & \sigma(T_c)+\sigma(T_dT_c^{-1}T_d^{-1}) \nonumber \\
& = & \sigma(T_c)+T_d \sigma(T_c^{-1}) \nonumber \\
& = & \sigma(T_c)+T_d \sigma(T_c) \nonumber 
\end{eqnarray}
Let $[d]=\alpha_1a_1+\cdots + \alpha_ga_g+\beta_1b_1+ \cdots +\beta_gb_g.$  Without loss of generality, suppose $\bar{a}_1\bar{b}_1$ is a term in $\sigma(T_c)$. 

Then using the fact that $[T_b^k(a)] = [a]+k \hat{i}(a,b)[b]$, we see
\begin{eqnarray}
T_d(\bar{a}_1\bar{b}_1) & = & T_d(\bar{a}_1)T_d(\bar{b}_1) \nonumber \\
& = &(\overline{T_d(a_1)})  (\overline{T_d(b_1)}) \nonumber \\ 
& = & (\overline{a_1+\beta_1[d]}) (\overline{b_1+\alpha_1[d]}) \nonumber
\end{eqnarray}
Hence the $\Z_2$ coefficients of the symplectic terms are:
$$\bar{a}_1 \bar{b}_1:   \alpha_1^2\beta_1^2 + (1+\beta_1\alpha_1)(1+\alpha_1\beta_1) = 1$$
$$\bar{a}_i \bar{b}_i:   (\beta_1\alpha_i)(\alpha_1\beta_i)+(\beta_1\beta_i)(\alpha_1\alpha_i) = 0,    \forall i : 1 < i \leq g$$
So in $\sigma([T_c,T_d])$ the $\bar{a}_1\bar{b}_1$ terms will cancel out.  Now suppose $\bar{a}_1\bar{b}_2$ is a term in $\sigma(T_c)$, similarly for any other nonsymplectic term.  Then
\begin{eqnarray}
T_d(\bar{a}_1\bar{b}_2) & = & (\overline{T_d(a_1)})(\overline{T_d(b_2)}) \nonumber \\
& = & (\overline{a_1 + \beta_1[d]})(\overline{b_2+\alpha_2[d]}) \nonumber
\end{eqnarray}
Again, we only need to consider the symplectic coefficients.
$$\bar{a}_1 \bar{b}_1:   (1+\beta_1\alpha_1)(\alpha_2\beta_1) + \beta_1^2(\alpha_2\alpha_1) = \alpha_2\beta_1$$
$$\bar{a}_2 \bar{b}_2:   (\beta_1\alpha_2)(1+\alpha_2\beta_2)+(\beta_1\beta_2)(\alpha_2^2) = \alpha_2\beta_1$$
$$\bar{a}_i \bar{b}_i:   (\beta_1\alpha_i)(\alpha_2\beta_i)+(\beta_1\beta_i)(\alpha_2\alpha_i) = 0,   \forall i : 2 < i \leq g $$
Notice the $\bar{a}_1 \bar{b}_1$ and $\bar{a}_2 \bar{b}_2$ coefficients are the same and the other symplectic terms have coefficient 0.  Hence we get a sum of two symplectic terms in $\sigma([T_c,T_d])$.  This is the case for any nonsymplectic term; the only way an $\bar{a}_i\bar{b}_i$ term will appear in $\sigma([T_c,T_d])$ is in a pair.

Suppose $\bar{a}_1$, or any other linear term, is in $\sigma(T_c).$  Then $T_d(\bar{a}_1) = (\overline{a_1+\beta_1[d]})$ which has no terms of degree two.  Similarly if $1$ is in $\sigma(T_c)$, then $T_d(1) = 1$ because the action of $\Mod(S)$ on $B_2$ is a linear isomorphism.  Thus all symplectic terms in $\sigma(\text{SSIP}(S))$ are in $\langle  \bar{a}_i \bar{b}_i+\bar{a}_j \bar{b}_j | 1 \leq i, j \leq g, i \ne j \rangle$.
\end{proof}

\section{Areas for further study}
In this paper we have done a preliminary investigation of SIP-maps and the group they generate, $\SIP(S)$.  We have factored SIP-maps into the product of 5 BP-maps, shown the image of SIP-maps under well-known representations of the Torelli group, and characterized which SIP-maps are in $\K(S)$ and the kernel of the Birman-Craggs-Johnson homomorphism.  Further we have shown $\SIP(S) \neq \I(S)$ and is in fact an infinite index subgroup when $g \geq 3$.  We have also given several equivalent descriptions of the group $\langle \SIP(S), \K(S) \rangle$.  We have also found a new interpretation in terms of SIP-maps of a relation among Johnson generators of $\I(S)$.  

This work leads to many questions about SIP-maps as well as the structure of $\SIP(S)$ that deserve further investigation.  For example:
\begin{itemize}
\item Is $\I(S) / \SIP(S)$ abelian?
\item Can other relations in $\I(S)$ be reinterpreted in terms of SIP-maps?
\item Is $\SIP(S)$ finitely generated?
\item Is $\SIP(S)$ finitely presentable?
\end{itemize}

\bibliographystyle{plain}
\bibliography{sts}

\end{document}